\documentclass[10pt,reqno,a4paper]{amsart}

\usepackage[pdftex]{graphicx}
\usepackage{pgfplots}
\usepackage{tikz}  
\usepackage{caption}

\usetikzlibrary{arrows.meta,decorations.pathmorphing,backgrounds,positioning,fit,petri}

\pgfplotsset{compat=1.18}

\usepackage{amsmath}
\usepackage{amssymb,amsthm,hyperref,caption}%,times
\usepackage{color}
\definecolor{meinBlau}{rgb}{0.2,0.2,0.9} %color for in-document links
\definecolor{blau}{rgb}{0,0,0.75} %color for in-document links
\definecolor{rot}{rgb}{0.74,0,0} %color for in-document links
\hypersetup{colorlinks,linkcolor=blau,citecolor=blue,urlcolor=meinBlau}
\usepackage{times}
\usepackage{enumitem}

\usepackage{xcolor}
\usepackage{tabularray}

\tikzset{
    declare function={
        normcdf(\x)= 1/(1 + exp(-0.07056*(\x)^3 - 1.5976*(\x)));
    }
}

\allowdisplaybreaks

\newtheorem{theorem}{Theorem}
\newtheorem{lem}[theorem]{Lemma}

\newtheorem{prop}{Proposition}

\theoremstyle{definition}

\newtheorem{remark}{Remark}
\newtheorem{example}{Example}

\def\P{{\mathbb {P}}}
\def\E{{\mathbb {E}}}

\newcommand{\fallfak}[2]{\ensuremath{#1^{\underline{#2}}}}

\newcommand{\N}{\ensuremath{\mathbb{N}}}
\newcommand{\R}{\ensuremath{\mathbb{R}}}
\DeclareMathOperator{\erf}{erf}

\newcommand{\be}{\ensuremath{\boldsymbol{1}}}
\newcommand{\bz}{\ensuremath{\boldsymbol{2}}}
\newcommand{\bd}{\ensuremath{\boldsymbol{3}}}
\newcommand{\bv}{\ensuremath{\boldsymbol{4}}}

\newcommand{\Stir}[2]{\genfrac{ \{ }{ \} }{0pt}{}{#1}{#2}}

\DeclareMathOperator{\law}{\overset{\mathcal{L}}{=}}
\DeclareMathOperator{\claw}{\overset{\mathcal{L}}{\rightarrow}}

\DeclareMathOperator{\LinExp}{LinExp}

\DeclareMathOperator{\Bin}{B}

\newcommand{\cut}{\ensuremath{\mathsf{Cut}}}

\newcommand{\entr}[3]
{
\ensuremath
	{
  \begin{matrix}#1\\#2\\#3\end{matrix}
	}
}

\newcommand{\entry}[4]
{
\ensuremath
	{
    \begin{matrix}#1\\#2\\#3\\#4\end{matrix}
	}
}

\begin{document}

\author[M.~Kuba]{Markus Kuba}
\address{Markus Kuba\\
Department Applied Mathematics and Physics\\
University of Applied Sciences - Technikum Wien\\
H\"ochst\"adtplatz 5, 1200 Wien} %
\email{kuba@technikum-wien.at}

\author[A.~Panholzer]{Alois Panholzer}
\address{Alois Panholzer\\
Institut f{\"u}r Diskrete Mathematik und Geometrie\\
Technische Universit\"at Wien\\
Wiedner Hauptstr. 8-10/104\\
1040 Wien, Austria} \email{Alois.Panholzer@tuwien.ac.at}

\title[Limit law for no feedback one-time riffle shuffle]{On Card guessing games: limit law for no feedback one-time riffle shuffle}


\begin{abstract}
We consider the following card guessing game with no feedback. An ordered deck of
$n$ cards labeled $1$ up to $n$ is riffle-shuffled exactly one time. 
Then, the goal of the game is to maximize the number of correct guesses of the cards. One after another a single card is drawn from the top, the guesser makes a guess without seeing the card and gets no response if the guess was correct or not.
Building upon and improving earlier results, we provide a limit law for the number of correct guesses and also show convergence of the integer moments. 
\end{abstract}

\maketitle

\section{Introduction}
The analysis of card shuffling and card guessing games has a long history. Starting from a mathematical model of shuffling developed
in 1956 at Bell Labs by E.~Gilbert and C.~Shannon, 
the subject has extended in various directions in a great many articles, amongst others~\cite{AldousDiaconis1986,Diaconis1978,DiaconisGraham1981,DiaconisMcPitman1995,HeOttolini2021,KnoPro2001,PK2023,KuPanPro2009,Leva1988,OttoliniSteiner2022,OT2023,Read1962,Zagier1990}. 
The mathematical analysis of questions related to card shuffling and card guessing are
not only of purely theoretical interest. There are applications to the analysis of clinical trials
~\cite{BlackwellHodges1957,Efron1971}, fraud detection related to 
extra-sensory perceptions~\cite{Diaconis1978}, guessing so-called Zener Cards~\cite{OttoliniSteiner2022}, 
as well as relations to tea tasting and the design of statistical experiments~\cite{Fisher1936,OT2023}.

\smallskip

In this work we consider the following problem. A deck of $n$ cards labeled consecutively
from 1 on top to $n$ on bottom is face down on the table. The deck
is riffle shuffled once and placed back on the table, face down. A
guesser tries to guess at the cards one at a time, starting from the top.  
The goal is to maximize the number of correct guesses with the caveat, that the
identities of the card guessed are not revealed, nor is the guesser told
whether a particular guess was correct or not. Such card guessing games 
are usually called \textit{no feedback} games, as no information at all is delivered to the person guessing. 
In contrast, there are \textit{complete feedback} games, where the guesser is shown the drawn card and thus knows,
whether the guess was correct or not. For a similar card guessing game with complete feedback we refer the reader to
\cite{KT2023,Liu2021}. 

\smallskip 

The optimal strategy for the no feedback game, as well as extensions to $k$-time riffle shuffles, has been given by Ciucu~\cite{Ciucu1998}.
Therein, he also derived the expected value $\E(X_n)$ of the number of correct guesses $X_n$, when the deck is riffle shuffled once; see also Krityakierne and Thanatipanonda~\cite{NFNW-KT2022} for related results. The first few higher moments of $X_n$ were derived by Krityakierne et al.~\cite{KSTY2023}. We build on the earlier work~\cite{Ciucu1998,KSTY2023} and derive in this article the limit law of the number of correct guesses $X_n$ in the no feedback game and also give asymptotic results for all integer moments, extending the results of~\cite{KSTY2023}. 

\smallskip 

Finally, we also comment on a different kind of card guessing games under the uniform distribution. 
A deck of a total of $M$ cards is shuffled, and then the guesser is provided with the total number of cards $M$, as well as the individual numbers of say hearts, diamonds, clubs and spades. After each guess, the person guessing the cards is shown 
the drawn card, which is then removed from the deck. This process is continued until no more cards are left. 
Assuming the guesser tries to maximize the number of correct guesses, one is again interested in the total number of correct guesses. 
The card guessing procedure can be generalized to an arbitrary number $N\ge 2$ of different types of cards. 
In the simplest setting there are two colors, red (hearts and diamonds) and black (clubs or spades), 
and their numbers are given by non-negative integers $m_1$, $m_2$, with $M=m_1+m_2$. One is then interested in the random variable $C_{m_1,m_2}$ counting the number of correct guesses. Interestingly, it turned out that the random variable $C_{m_1,m_2}$ is closely 
related to card guessing with complete feedback after a single riffle shuffle. For this complete feedback card guessing game, not only the expected value and the distribution of the number of correct guesses is known~\cite{DiaconisGraham1981,KnoPro2001,Leva1988,Read1962,Zagier1990}, but also multivariate limit laws and interesting relations to combinatorial objects such as Dyck paths and P\'olya-Eggenberger urn models have been established~\cite{DiaconisGraham1981,PK2023,KuPanPro2009}. 

\subsection{Notation}
As a remark concerning notation used throughout this work, we always write $X \law Y$ to express equality in distribution of two random variables (r.v.) $X$ and $Y$, and $X_{n} \claw X$ for the weak convergence (i.e., convergence in distribution) of a sequence of random variables $X_{n}$ to a r.v.\ $X$. Furthermore, throughout this work we let $h:=h(n)=\lceil \frac{n}{2}\rceil$. 
Moreover, we denote for $s\in\N$ with $\fallfak{x}{s}=x(x-1) \cdots (x-(s-1))$ the falling factorials.

\section{Riffle shuffle model and optimal strategy\label{Subsection_GSR}}
\subsection{Gilbert-Shannon-Reeds model}
The riffle shuffle, sometimes also called dovetail shuffle, is a card shuffling technique. In the mathematical modeling of card shuffling, the \emph{Gilbert–Shannon–Reeds} model~\cite{DiaconisGraham1981,Gilbert1955} is a probability distribution serving as a model of a riffle shuffle. 
One considers a sorted deck of $n$ cards labeled consecutively from 1 up to $n$. 
The deck of cards is cut into two packets, assuming that
the probability of selecting $k$ cards in the first packet, which we call a cut at position $k$, and $n-k$ in the second packet is defined as a binomial distribution with parameters $n$ and $1/2$:
\[
\P\{\cut=k\}=
\frac{\binom{n}k}{2^n},\quad 0\le k\le n.
\]
Afterward, the two packets are interleaved back into a single pile: one card at a time 
is moved from the bottom of one of the packets to the top of the shuffled deck, such that if $m_1$ cards remain in the first and $m_2$ cards remain in the second packet, then the probability of choosing a card from the first packet is $m_1/ ( m_1 + m_2 )$ and the probability of choosing a card from the second packet is $m_2 / ( m_1 + m_2 )$. 
\begin{figure}[!htb]
\includegraphics[scale=0.55]{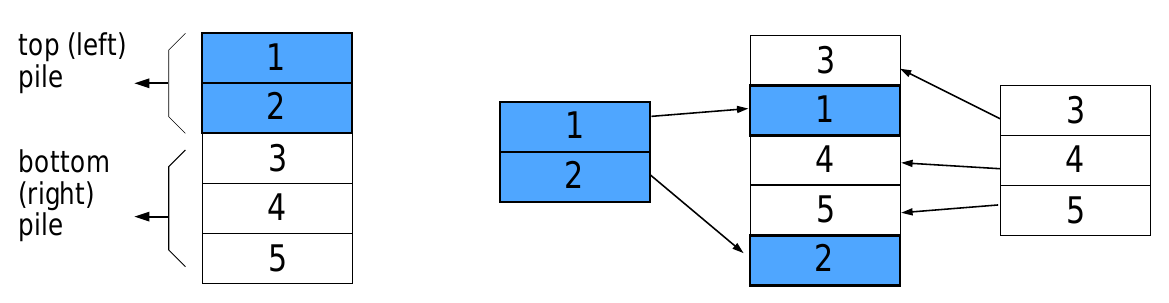}%
\caption{Example of a one-time riffle shuffle: a deck of five cards is split after 2 with probability $\binom52/2^5=5/16$ and then interleaved.}%
%\label{}%
\end{figure}
For a one-time shuffle, the operation of interleaving described above gives rise to an ordered
deck (corresponding to the identity permutation) with multiplicity $n + 1$. 
Each other shuffled deck corresponds to a permutation with exactly two increasing subsequences
and has multiplicity 1. In total number there are $2^n- n - 1$ different permutations of two increasing
subsequences from the interleaving.

\subsection{Optimal strategy}
The optimal strategy $\mathcal{G}^{\ast}$ for maximizing the number $X_n$ of correctly guessed cards, starting with a deck of $n$ ordered cards, after a one-time riffle shuffle stems from the following proposition based on work of Ciucu~\cite{Ciucu1998}
and Krityakierne and Thanatipanonda~\cite{NFNW-KT2022}. 
\begin{prop}[\cite{Ciucu1998,NFNW-KT2022}]
\label{Prop:firstCard}
In order to maximize the number of correct guesses in a one-time riffle shuffle no feedback card guessing game, for large $n$, the guesser should follow the optimal strategy $\mathcal{G}^{\ast}$: guess the top half of the deck
with sequence
\[
1,2,2,3,3,4,4,\dots
\]
and guess the bottom half of the deck with sequence
\[
\dots,n-3,n-3,n-2,n-2,n-1,n-1,n.
\]
\end{prop}
\begin{proof}
For the sake of completeness and to make this work more self-contained, we add the nice and short argument justifying this strategy. From the Gilbert-Shannon-Reeds model one can readily determine the probability $m_{i,j}=m_{i,j}(n)$ that the card labeled $i$ ends up at position $j$ after a riffle shuffle, starting with a deck of $n$ ordered cards:
\begin{equation}
\begin{split}
m_{i,i}&=\frac{1}{2^n}\big(2^{i-1}+2^{n-i}\big),\\
m_{i,j}&=\frac{1}{2^{n-j+1}}\binom{n-j}{i-j},\quad j<i, 
\end{split}
\end{equation}
and the symmetry
\[
m_{i,j}=m_{n-i+1,n-j+1}.
\]
This follows directly by considering the different cutting positions $k$, $0\le k\le n$, and the number of different interleavings,
such that card labeled $i$ ends up at position $j$. Let $j<i$. Then, cuts at positions $k\ge i$ cannot contribute, 
as all cards labeled $1$ up to $i-1$ are still before $i$ after interleaving and the final position is at least $i$. Thus, we are left with cuts at $1\le k<i$. 
As the cards $k+1$ up to $i-1$, a total of $i-k-1$, of the second packet are always before $i$, we require exactly $j-i+k$ of the first packet out of the cards labeled $\{1,\dots,k\}$ to be interleaved
before $j$. There are $\binom{j-1}{j-i+k}$ ways to do so. The remaining $i-j$ cards of the first packet have to be interleaved above $j$, leading to $\binom{n-j}{i-j}$ different ways. 
In total,
\[
m_{i,j}=\frac{1}{2^n}\sum_{k=i-j}^{i-1}\binom{j-1}{j-i+k}\binom{n-j}{i-j}=\frac1{2^{n-j+1}}\binom{n-j}{i-j}.
\]
In the case $i=j$ there are additional contributions from the cuts at $k\ge i$, leading
to 
\[
m_{i,i}%=\frac1{2^{n-i+1}}+\sum_{k=i}^{n}\frac{\binom{n}{k}}{2^n}\cdot \binom{n}{n-i}\cdot\frac{k!}{(n-k)!}{n!}
=\frac1{2^{n-i+1}}+\sum_{k=i}^{n}\frac{\binom{n-i}{k-i}}{2^n}
=\frac1{2^{n-i+1}}+\frac1{2^{i}}.
\]
The case $j>i$ can be treated similarly; we opt to recall a nice symmetry argument of \cite{Ciucu1998}:
We imagine having a second set of numbers on our cards, in which the cards are labeled consecutively from 1 on bottom through $n$ on top.
We call this the ''upward labeling'', compared to the original ''downward labeling''.
It is clear that, after a riffle shuffle, card $i$ ends up in position $j$ in
downward labeling if and only if card $n-i+1$ goes to position $n-j+1$ in
the upward labeling. Since the probability distributions involved in the riffle
shuffle have a vertical symmetry axis, we obtain the stated symmetry. 
Finally, the best guess $g_{j}$ at the card in position $j$ 
of the optimal strategy $\mathcal{G}^{\ast}=g_1g_2\dots g_n$
is determined by guessing the asymptotically largest probability,
\[
g_{j}=\max_{i}\{m_{i,j}\}, \quad 1\le j\le n,
\]
which can be obtained by a close inspection of the binomial coefficients.
\end{proof}

\begin{remark}
As already pointed out in~\cite{NFNW-KT2022}, the optimal strategy is not unique for a one-time riffle shuffle.
In particular, for the card position $j$, the player can optimally choose to guess any number from the set $\mathcal{S}_{j}$, where
$\mathcal{S}=(\mathcal{S}_{j})$:
\[
\mathcal{S}= \{1\}, \{2\}, \{2\}, \{2, 3\}, \{3\}, \{3, 4\}, \{4\}, \{4, 5\},\dots\quad\text{top half}
\]
and 
\[
\dots, \{n-3,n-2\},\{n-2\},\{n-2,n-1\},\{n-1\},\{n-1\},\{n\}\quad\text{bottom half}.
\]
However, in our analysis we follow exclusively the strategy $\mathcal{G}^{\ast}$, which is the one that can be extended naturally to multiple-time riffle shuffle~\cite{NFNW-KT2022}.
\end{remark}

\begin{example}
We consider the case $n=3$ and the $2^3$ possible permutations $\sigma$. In Table~\ref{Table1} we highlight all $2^3-3=5$ different permutations, colored cyan, as well as the cut positions and number of correct guesses.

\begin{table}[!htb]
%\setcellgapes{.6mm}  \makegapedcells  %<- this does nice spacing, but not compatible with rowcolor
\begin{tblr}{
colspec={||c||c||c|c|c|| c|c|c|| c||},
cell{1}{2}={cyan9},cell{1}{4}={cyan9},cell{1}{5}={cyan9},cell{1}{7}={cyan9},cell{1}{8}={cyan9},
}
\hline
%\rowcolor{grau}
$\sigma$&  \entr\be\bz\bd & \entr\be\bz\bd&\entr21\bd&\entr231  &  \entr\be\bz\bd&\entr\be32&\entr312 & \entr\be\bz\bd\\%[2mm]
\hline
\cut &0  &  1&1&1  &  2&2&2   & 3\\ 
\hline
$X$ & 3 & 3&1&0  & 3&1&0 &3\\
\hline
\end{tblr}
\caption{Case $n=3$: optimal strategy $\mathcal{G}^{\ast}=(1,2,3)$ and the number of correct guesses.}
\label{Table1}
\end{table}

We observe that under the optimal strategy we have
\[
\P\{X_3=3\}=\frac12,\quad\P\{X_3=2\}=0, \quad\P\{X_3=1\}=\P\{X_3=0\}=\frac14.
\]
\end{example}

\begin{example}
We consider the case $n=4$ and the $2^4$ possible permutations $\sigma$. Again, in Table~\ref{Table2} we highlight all $2^4-4=12$ different permutations, colored cyan, as well as the cut positions and number of correct guesses.
\renewcommand{\arraystretch}{0.7}
{\footnotesize
\begin{table}[!htb]
%\setcellgapes{.6mm}  \makegapedcells  %<- this does nice spacing, but not compatible with rowcolor
\begin{tblr}{colsep={4pt},
colspec={||c || c || c|c|c|c || c|c|c|c|c|c || c|c|c|c ||  c ||},
cell{1}{2}={cyan9},cell{1}{4}={cyan9},cell{1}{5}={cyan9},cell{1}{6}={cyan9},cell{1}{8}={cyan9},cell{1}{9}={cyan9},
cell{1}{10}={cyan9},cell{1}{11}={cyan9},cell{1}{12}={cyan9},cell{1}{14}={cyan9},cell{1}{15}={cyan9},cell{1}{16}={cyan9},
}
\hline
%\rowcolor{grau}
$\sigma$ &  \entry\be\bz\bd\bv &   \entry\be\bz\bd\bv&\entry21\bd\bv&\entry231\bv&\entry2341  &  \entry\be\bz\bd\bv&\entry\be32\bv&\entry\be342&\entry312\bv&\entry3142&\entry3412 & \entry\be\bz\bd\bv
& \entry\be\bz43 & \entry\be423 & \entry4123 & \entry\be\bz\bd\bv
\\%[2mm]
\hline
\cut &0  &  1&1&1&1  &  2&2&2&2&2&2   & 3&3&3&3 & 4\\ 
\hline
$X$ & 4 & 4&2&1&0  & 4&2&1&1&0&0 &4&2&1&0  &4\\
\hline
\end{tblr}
\caption{Case $n=4$: optimal strategy $\mathcal{G}^{\ast}=(1,2,3,4)$ and the number of correct guesses.}
\label{Table2}
\end{table}
}

Under the optimal strategy we obtain
\[
\P\{X_4=4\}=\frac5{16},\,\P\{X_4=3\}=0, \,\P\{X_4=2\}=\frac3{16}, \,\P\{X_4=1\}=\P\{X_4=0\}=\frac14.
\]
\end{example}

We further simulated the probability mass function by looking at the empirical probabilities $h_k$ for $n=200$, $1000$ and $5000$ with samples of size $N=50000$ for $n=200$, $1000$ and sample size $N=100000$ for $n=5000$.

\begin{center}
 \begin{tikzpicture}
        \begin{axis}[height=5cm,	width=4.7cm]%, xlabel={$k$}, ylabel={$h_k$}
        \addplot table [x=x, y=y, col sep=comma] {200-50000.csv};
				\end{axis}
 \end{tikzpicture}
\quad
 \begin{tikzpicture}
        \begin{axis}[height=5cm, width=4.7cm]
        \addplot table [x=x, y=y, col sep=comma] {1000-50000.csv};
				\end{axis}
 \end{tikzpicture}
\quad
 \begin{tikzpicture}
        \begin{axis}[height=5cm, width=4.7cm]
        \addplot table [x=x, y=y, col sep=comma] {5000-100000.csv};
        \end{axis}
 \end{tikzpicture}
\end{center}

\section{Distributional analysis and generating functions}
Let $X_n$ denote the random variable counting the number of correct guesses under the optimal strategy $\mathcal{G}^{\ast}$ in the no-feedback model after a single riffle shuffle, starting with $n$ ordered cards. In~\cite{KSTY2023}, the distribution of $X_n$ has been determined using
the generating function $f_n(q)$:
\[
\E(q^{X_n})=\frac{f_n(q)}{2^n}.
\]
Our starting point is the following nice result, giving a formula for $f_n(q)$ in terms of an auxiliary generating function $g_{m_{1},m_{2}}(q)$, which is described itself in a recursive way. 
\begin{lem}[Krityakierne et al.~\cite{KSTY2023}]
\label{lem:no1}
Let $h:=\lceil \frac{n}{2} \rceil$. The generating function $f_n(q)$ satisfies
\[
f_n(q)=4q^4-2(q^2+q^3)+\sum_{a=0}^{h}\sum_{b=0}^{n-h}g_{a,h-a}(q) \cdot g_{b,n-h-b}(q).
\]
Here, the generating function $g_{m_1,m_2}(q)$ is determined by the recurrence relation
\begin{equation}
\label{eqn:no1}
g_{m_1,m_2}(q)=q^{\delta(c,m_1)}g_{m_1-1,m_2}(q)+g_{m_1,m_2-1}(q),
\quad m_1,m_2\ge 0 \; \text{and} \; (m_{1},m_{2}) \neq (0,0),
\end{equation}
where $c=\lfloor \frac{m_1+m_2}{2}\rfloor+1$ and $\delta(x,y)$ denotes the Kronecker delta function, with initial values $g_{0,0}(q)=1$ and $g_{m_{1},m_{2}}(q) = 0$, for $m_{1}<0$ or $m_{2}<0$.
\end{lem} 

In order to study the limit law of $X_n$, we interpret the results of Lemma~\ref{lem:no1} 
in a probabilistic way. Let $Y_{m_1,m_2}$ denote the random variable 
defined in terms of recurrence relation~\eqref{eqn:no1},
\begin{equation}
\label{eqn:no2}
\E(q^{Y_{m_1,m_2}})=\frac{g_{m_1,m_2}(q)}{g_{m_1,m_2}(1)}=\frac{g_{m_1,m_2}(q)}{\binom{m_1+m_2}{m_1}},
\end{equation}
where the latter equality holds, since for $q=1$ recurrence~\eqref{eqn:no1} is equivalent to Pascal's rule for the binomial coefficients.
Next we translate above recurrence relations into a distributional equation for $X_n$. 
\begin{prop}[Distributional equation for $X_n$]
\label{prop1}
The random variable $X_n$ satisfies for $n\to\infty$:
\[
X_n\sim Y_{h-J_h,J_h}+Y^{\ast}_{n-h-J^{\ast}_{n-h},J^{\ast}_{n-h}},
\]
where $Y$, $Y^{\ast}$ are distributed according to~\eqref{eqn:no2}, $J$, $J^{\ast}$ are binomially distributed with parameter $p=1/2$
and parameters $h:=\lceil \frac{n}{2} \rceil$ and $n-h$, respectively. Moreover, all random variables are independent. 
\end{prop}
\begin{proof}
By Lemma~\ref{lem:no1} we have
\begin{align*}
\E(q^{X_n})&=\frac{4q^4-2(q^2+q^3)}{2^n}+\frac{1}{2^n}\sum_{a=0}^{h}\sum_{b=0}^{n-h}g_{a,h-a}(q) \cdot g_{b,n-h-b}(q)\\
&=\frac{4q^4-2(q^2+q^3)}{2^n}+\sum_{a=0}^{h}\frac{g_{a,h-a}(q)}{\binom{h}{a}}\cdot\frac{\binom{h}{a}}{2^h}\sum_{b=0}^{n-h}\frac{g_{b,n-h-b}(q)}{\binom{n-h}{b}}\cdot\frac{\binom{n-h}{b}}{2^{n-h}}\\
&=\frac{4q^4-2(q^2+q^3)}{2^n}\\
& \qquad \mbox{} +\sum_{a=0}^{h}\E(q^{Y_{a,h-a}})\cdot \P\{J_h=a\}\sum_{b=0}^{n-h}\E(q^{Y^{\ast}_{b,n-h-b}})\P\{J^{\ast}_{n-h}=b\}.
\end{align*}
Using that the product of probability generating functions corresponds to a sum of independent random variables, we obtain the distributional equation
\begin{equation*}
  X_{n} \law \hat{Y}_{h-J_{h},J_{h}} + \hat{Y}^{\ast}_{n-h-J^{\ast}_{n-h},J^{\ast}_{n-h}},
\end{equation*}
where the "hat" versions differ from their ordinary versions only on the values $\{2,3,4\}$, with the difference tending to zero exponentially fast for $n \to \infty$. Thus, we can safely neglect this difference when characterizing the limiting behaviour.
\end{proof}

\subsection{Limit law} 
From the properties of the binomial distribution
we know that $J_h=\Bin(h,\frac12)$ satisfies
\begin{equation}\label{eqn:binomial_approximation}
J_h\sim \mu_J + \sigma_J\cdot \mathcal{N},
\end{equation}
with $\mu_J =\frac{h}{2}\sim \frac{n}{4}$, $\sigma_J=\frac{\sqrt{h}}{2}\sim \frac{\sqrt{n}}{2\sqrt{2}}$
and $\mathcal{N}=\mathcal{N}(0,1)$ the standard normal distribution. In view of Proposition~\ref{prop1} this implies that we need to know the distribution
of $Y_{m_1,m_2}$ with parameters 
\[
m_1\sim h-\mu_j-\sigma_J\cdot t,\quad m_2\sim \mu_j+\sigma_J\cdot t,\quad t\in\R.
\]
Thus we require the limit law of $Y_{m_1,m_2}$, for $m_1,m_2\to\infty$ and satisfying the assumptions above. In order to analyze recurrence relation~\eqref{eqn:no1} and thus $Y_{m_1,m_2}$, we proceed similar to the two-color card guessing game~\cite{PK2023}, setting up a suitable bijection. This bijection allows to analyze $Y_{m_1,m_2}$ in terms of certain Dyck paths by using tools from Analytic Combinatorics~\cite{FlaSed}. 

\smallskip

First, according to recurrence~\eqref{eqn:no1} we consider the sample paths from $(m_1,m_2)$ to $(0,0)$, $m_1,m_2\ge 0$, with steps $(-1,0)$, ``left'', and $(0,-1)$, ``down'', where the leftward steps carry a weight $q$ if $c=\lfloor \frac{m_1+m_2}{2}\rfloor+1$. 
Next we reverse the direction of the walks, thus going from $(0,0)$ to $(m_1,m_2)$, and then rotate the coordinate system clockwise by 45 degrees. 
After scaling, the resulting walks are directed walks of length $m_1+m_2$ with Dyck steps $(1,1)$, ``upward'', and $(1,-1)$, ``downward'', starting at the origin and ending at $(m_{1}+m_{2},m_{2}-m_{1})$, see Figure~\ref{pic1}.

\begin{figure}[!htb]
\includegraphics[scale=0.5]{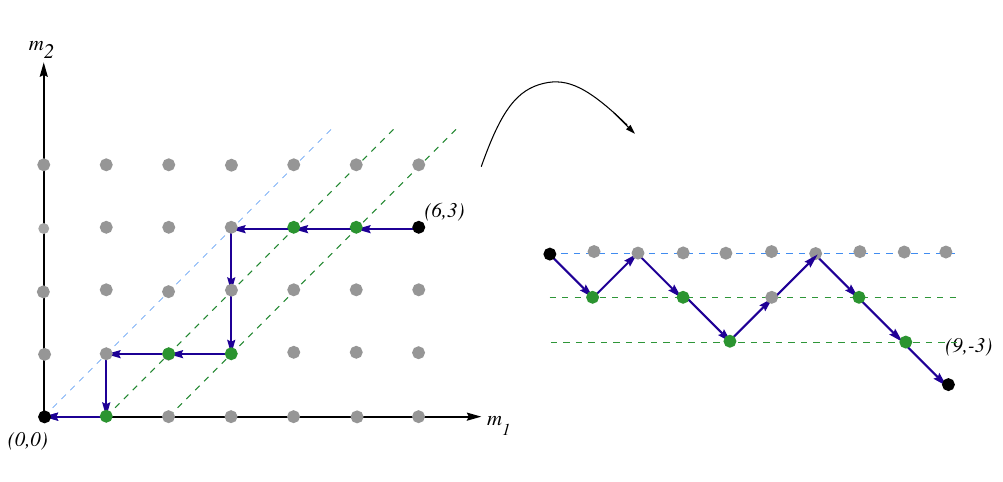}%
\caption{Mapping of a sample path of $Y_{6,3}$ to a directed lattice path from the origin to $(9,-3)$.}%
\label{pic1}%
\end{figure}

It remains to translate the weight $q^{\delta(c,m_1)}$, with $c=\lfloor \frac{m_1+m_2}{2}\rfloor+1$, to the directed paths. We consider all four different cases resulting from the parity of $m_{1},m_{2}$, where we obtain the following.
\begin{itemize}
\item Both $m_1,m_2$ are even or both $m_1,m_2$ are odd:
\[
\Big\lfloor \frac{m_1+m_2}{2}\Big\rfloor+1=\frac{m_1+m_2}{2}+1=m_1,\quad \text{which yields} \quad m_1=m_2+2.
\]
\item $m_1$ is even and $m_2$ odd or vice versa:
\[
\Big\lfloor \frac{m_1+m_2}{2}\Big\rfloor+1=\frac{m_1+m_2-1}{2}+1=m_1,\quad \text{which yields} \quad m_1=m_2+1.
\]  
\end{itemize}
This implies that $q$ counts the number of contacts with the two lines
\[
g_1\colon \ y=x-1, \quad g_2\colon\ y=x-2.
\]
After rotation, i.e., for Dyck paths, this implies that we count contacts with the lines $y=-1$, corresponding to $g_1$, as well as $y=-2$, corresponding to $g_2$.
However, as only left steps from $(m_1,m_2)\to(m_1-1,m_2)$ can carry a weight in the original sample paths, in the Dyck path setting a contact
with these two lines is only counted when occurring after a downward step, see Figure~\ref{pic2}.
Above findings are summarized as follows.

\begin{figure}[!htb]
\includegraphics[scale=0.5]{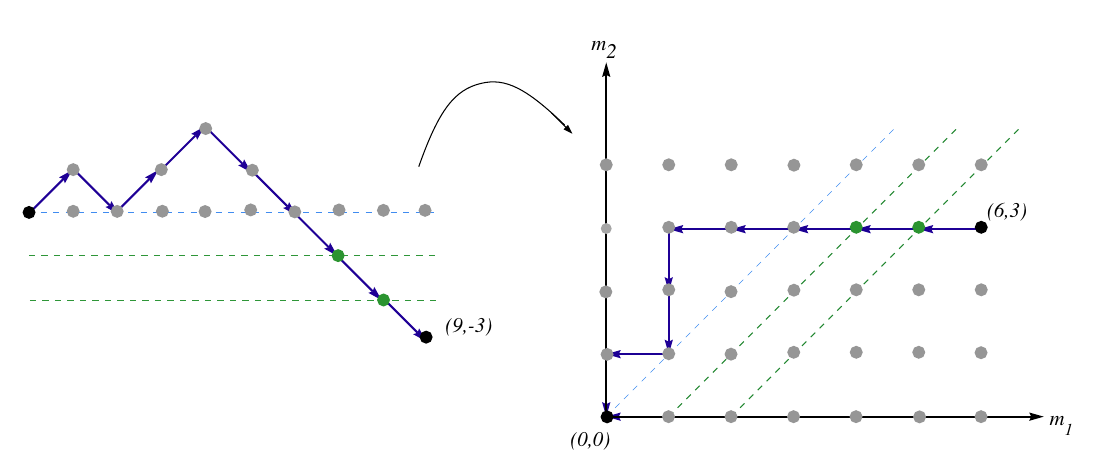}%
\caption{Mapping of a directed lattice path from the origin to $(9,-3)$ to its corresponding sample path of $Y_{6,3}$.}%
\label{pic2}%
\end{figure}

\begin{prop}[Sample paths of the card guessing game and Dyck paths]
Let $\mathcal{S}_{m_1,m_2}$ denote the set of weighted sample paths from $(m_1,m_2)$ to $(0,0)$, $m_1,m_2\ge 0$, with steps $(-1,0)$,
carrying a weight $q$ if $c=\lfloor \frac{m_1+m_2}{2}\rfloor+1$, and $(0,-1)$. Then, 
$\mathcal{S}_{m_1,m_2}$ is in bijection with the set $\mathcal{D}_{m_1+m_2}$ of Dyck paths with step sets $(1,1)$ and $(1,-1)$ of length $m_1+m_2$, starting at the origin and ending at $(m_1+m_2,m_2-m_1)$, where the contacts with $y=-1$ and $y=-2$ are counted after a downward step.
\end{prop}

\begin{prop}
\label{prop:LinExp}
Assume that the numbers $m_1$, $m_2$ satisfy $m_1-m_2\sim t\cdot \sqrt{m_1}$, as $m_1\to\infty$, with $t>0$. Then, the random variable $Y_{m_1,m_2}$ is asymptotically linear exponentially distributed,
\[
\frac{Y_{m_1,m_2}}{\sqrt{m_1}}\claw \LinExp(\textstyle{\frac{t}{2}},\textstyle{\frac{1}{2}}),
\]
or equivalently by stating the cumulative distribution function,
\[
\P\{Y_{m_1,m_2}\le z\sqrt{m_1}\}\to 1-e^{-\frac{z(2t+z)}4}, \quad z \ge 0.
\]
An analogous result holds for $m_2-m_1\sim t\cdot \sqrt{m_1}$, as $m_2\to\infty$, with $t>0$.
\end{prop}

Before we prove this result, we discuss a motivation or, in other words, a back of the envelope explanation for it. Asymptotically, the number of down-contacts at levels $-1$ and $-2$ should be indistinguishable from the number of $W_{m_1,m_2}$ of returns to zero of Dyck paths of length $m_1+m_2$, starting at zero and ending at $m_2-m_1$. 
This latter quantity has been already analyzed in the context of two-color card guessing games, leading exactly to the same limit law. 
\begin{lem}[\cite{PK2023}]
For $m_{2} \sim m_{1}$, where the difference $d=m_{1}-m_{2}$ satisfies $d \sim t \sqrt{m_{1}}$, with $t > 0$: suitably scaled, $W_{m_{1},m_{2}}$ weakly converges to a linear exponential distribution, which is characterized via the distribution function
\begin{equation*}
  F(x) = 1-e^{-\frac{x(2t+x)}{4}}, \quad x > 0.
\end{equation*}
Thus, $W_{m_{1},m_{2}}$ is asymptotically linear exponentially distributed:
\begin{equation*}
  \frac{W_{m_1,m_2}}{\sqrt{m_{1}}} \claw \LinExp\big(\textstyle{\frac{t}{2}},\textstyle{\frac{1}{2}}\big).
\end{equation*}
\end{lem}

\smallskip

Next we show that the random variables $W_{m_1,m_2}$ and $Y_{m_1,m_2}$ are asymptotically indistinguishable. We show that both generating functions are almost identical and lead to the same asymptotic behavior. We derive the generating function of the weighted Dyck paths of interest, counting the number of down-contacts at levels $-1$ and $-2$ . We follow the classical analysis of Banderier and Flajolet~\cite{BanderierFlajolet2002}, see also~\cite{FlaSed}.
Assume that $m_1 \ge m_2+2$, such that the endpoint has a negative $y$-coordinate: $m_2-m_1 \le -2 <0$. 
We can decompose these Dyck paths into three parts (see Figure~\ref{fig:arch_decomposition}): 
\begin{itemize}
	\item Part one. An excursion $\mathcal{E}$ of arbitrary length, starting at the origin and never going below the $x$-axis and eventually returning to $y=0$, 
then followed by a downward step with additional weight $q$ leading to the first contact with $y=-1$.
\item  Part two. The middle part $\mathcal{M}$ consists itself of two different subparts. The first part consists of excursions $\mathcal{E}_{\text{up}}(q)$ initially going upwards from $y=-1$ to the same level. Here, the last step is always a downward step to $y=-1$ and is weighted with $q$. Second, we consider excursions $\mathcal{E}_{\text{down}}(q)$ going downwards, with an additional weight $q$ at the first step at height $y=-2$. The middle part finishes with a final contact at $y=-1$, followed by a downward step to $y=-2$ with additional weight $q$, yielding the formal description (see~\cite{FlaSed} for basic combinatorial constructions and the symbolic method):
\[
\mathcal{M}=\text{\textsc{Seq}}(\mathcal{E}_{\text{up}}(q) \cup\mathcal{E}_{\text{down}}(q))\times \{q\}.
\] 
\item Part three. An excursion starting at $y=-2$ and ending at $y=m_2-m_1 \le -2$, never reaching $y=-1$.
\end{itemize}
 
\begin{figure}[!htb]
\includegraphics[scale=0.2]{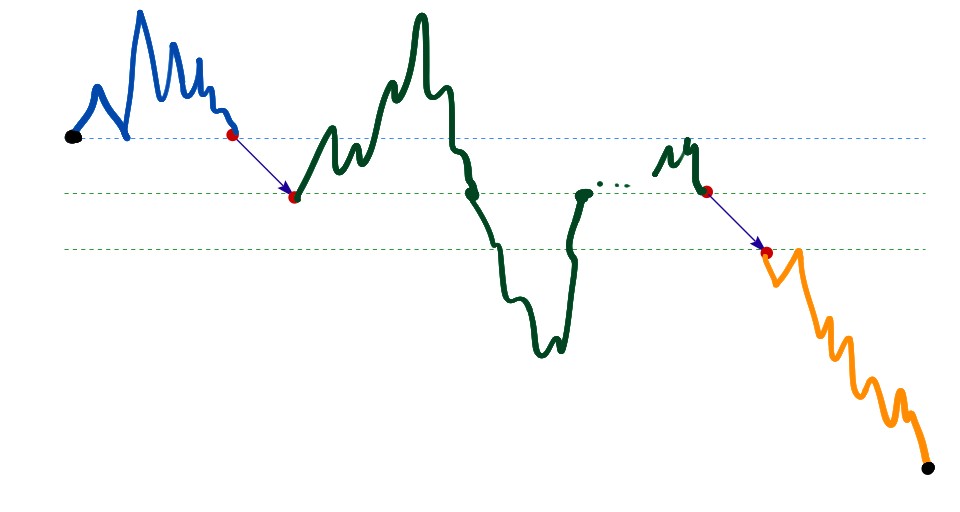}%
\caption{Visualization of the decomposition for $m_1\gg m_2$: the first part is the excursion (blue), 
the center part starts at $y=-1$ and ends at the same level (green), the final part departs from $y=-2$ to the ending position at
level $y=m_2-m_1\le -2$ (orange).}%
\label{fig:arch_decomposition}%
\end{figure}

By the classical arch decomposition, the generating function of excursions $E(z)$, 
obtained by using the sequence construction from the symbolic method in combinatorics~\cite{FlaSed}, 
satisfies
\[
E(z)=\frac{1}{1-z^2E(z)},\quad \text{which implies} \quad E(z)=\frac{1-\sqrt{1-4z^2}}{2z^2}.
\]
Thus, our searched function for part one equals $G_{1}(z) := E(z)\cdot zq$.
Furthermore, the generating function of excursions $\mathcal{E}_{\text{up}}(q)$ and $\mathcal{E}_{\text{down}}(q)$,
both with generating functions
\[
\frac{1}{1-z^2q E(z)} = 1 + \sum_{k \ge 1} \big(z^{2}qE(z)\big)^{k},
\]
are grouped together using the decomposition by arches corresponding to contacts
with $y=-1$, and taking into account the last additional downward step reaching $y=-2$ yields for part two:
\[
G_{2}(z) := M(z)=\frac{1}{1-2z^2qE(z)}\cdot zq.
\]
Finally, the generating function of the last excursion from $y=-2$ to $y=m_2-m_1$ is obtained 
using the classical generating function $F(z,u) = \sum_{\text{$p$ meander}} z^{\text{length of $p$}} u^{\text{final altitude of $p$}}$ of the final altitude of meanders, see Banderier and Flajolet~\cite{BanderierFlajolet2002}.
It is known that for Dyck paths it holds:
\[
F(z,u)=\frac{u-z E(z)}{u\big(1-z\big(u+\frac{1}u\big)\big)},
\]
where it is here assumed that the paths are starting at the origin and are never going below the $x$-axis. In our case, the paths
go from $y=-2$ to $y=m_2-m_1$, so by symmetry we need to extract the coefficient of $u^{m_1-m_2-2}$, thus we get for the third part $G_{3}(z) = [u^{m_{1}-m_{2}-2}] F(z,u)$.

Collecting all parts $G_{1}(z)$, $G_{2}(z)$ and $G_{3}(z)$, we obtain for the generating function 
\begin{equation*}
  G(z,q) = \sum_{\text{$p$ Dyck path}} z^{\text{length of $p$}} q^{\text{$\#$ downward visits at $y \in \{-1,-2\}$ of $p$}},
\end{equation*}
where the sum is running over all Dyck paths $p$ starting at the origin and ending at $y=m_{2}-m_{1} \le -2$, 
\begin{equation}
\label{eqn:u}
G(z,q)=[u^{m_1-m_2-2}]\frac{z^2q^2E(z)}{1-2z^2qE(z)}\cdot \frac{u-zE(z)}{u\big(1-z\big(u+\frac{1}u\big)\big)}.
\end{equation}
The extraction of coefficients can actually be carried out in an explicit manner, 
as the denominators factors nicely~\cite{BanderierFlajolet2002},
\[
u\big(1-z\big(u+\frac{1}u\big)\big)
=u-zu^2-z=-z(u-u_1(z))(u-u_2(z)),
\]
where
\begin{equation}\label{eqn:u1_u2}
u_1(z)=\frac{1-\sqrt{1-4z^2}}{2z}=zE(z),\quad u_2(z)=\frac{1+\sqrt{1-4z^2}}{2z},
\end{equation}
such that
\[
\frac{u-zE(z)}{u\big(1-z\big(u+\frac{1}u\big)\big)}
= \frac{1}{-z(u-u_2(z))}
=\frac{1}{z u_2(z)\big(1-\frac{u}{u_2(z)}\big)}
=\frac{u_1(z)}{z\big(1-u \cdot u_1(z)\big)},
\]
where we used that $u_1(z)u_2(z)=1$. This leads to the following result.
\begin{prop}
Assume that $m_1\ge m_2+2$. The generating function $G(z,q)$ of the number of Dyck paths of length $n=m_1+m_2$, starting at level $0$ and ending at level $m_2-m_1$, weighted according to downward visits at levels $y=-1$ and $y=-2$, is given by
\begin{equation}
\label{eqn:u2}
G(z,q) %= \frac{z q^{2} E(z)}{1-2z^2qE(z)} \cdot \big(z E(z)\big)^{m_1-m_2-1}.
= q^{2}\frac{1}{1-2z^2qE(z)} \cdot \big(z E(z)\big)^{m_1-m_2}.
\end{equation}
\end{prop}

\begin{remark}[Generalized composition schemes]
The structure of the generating function is similar to (extended) composition schemes
\[
F(z,q)=\Psi\big(q H(z)\big)\cdot M(z),
\]
considered in~\cite{BKW2024,BFSS2001,FlaSed}. The main novelty here is the dependance on the additional parameter $u$ in~\eqref{eqn:u}, or equivalently, the power $\big(z E(z)\big)^{m_1-m_2-1}$, where $m_1-m_2$ is allowed to depend on $n$, when analyzing $[z^n]F(z,q)$. 
This leads to new asymptotic regimes. A general study of such augmented schemes is forthcoming, together with the authors of~\cite{BKW2024}. 
\end{remark}

Next, we compare this generating function to the generating function of the number of zero-contacts starting at $y=0$ and ending at $y=m_2-m_1\neq 0$, where the walks have the length $n=m_1+m_2$. We can decompose this generating function
into two parts: the first part is a bridge, where the weight $q$ encodes the zero-contacts. Then, after a single step, we leave the $x$-axis and start our approach to the final position at $d=|m_2-m_1|$:
\begin{equation}\label{eqn:gf_zero_contacts}
\begin{split}
&[u^{|m_1-m_2|-1}]\frac{1}{1-2z^2qE(z)}\cdot z\cdot \frac{u-zE(z)}{u\big(1-z\big(u+\frac{1}u\big)\big)}\\
&\quad =\frac1z\cdot \frac{1}{1-2z^2qE(z)}\big(z E(z)\big)^{|m_1-m_2|+1}.
\end{split}
\end{equation}
We observe that both generating functions~\eqref{eqn:u} and \eqref{eqn:gf_zero_contacts} are nearly identical, except
a shift of length one in the length and in the difference $d=|m_1-m_2|$.
Symbolically, 
\[
Y_{m_1,m_2}\sim W_{m_1-1,m_2} +2,
\]
which readily leads to the stated limit law for $Y_{m_1,m_2}$. Note that one can give a more detailed analysis by extraction of coefficients, 
very similar to~\cite{PK2023}; we leave the details to the interested reader.

\smallskip

Now are are ready to state the main result, namely the limit law for $X_n$.

\begin{theorem}
\label{the:limitLaw}
The random variable $X_n/\sqrt{n}$, counting the normalized number of correct guesses
in a one-time riffle shuffle with no feedback card guessing game, tends for $n\to\infty$ to a sum of two independent identically distributed half-normal distributed random variables $H_1$, $H_2$, each one
with density function $f_{H}(x) = \frac{2}{\sqrt{\pi}}e^{-x^2}$, $x \ge 0$,
\[
\frac{X_n}{\sqrt{n}}\claw H_1 + H_2.
\]
Equivalently, the density $f(x):=f_{H_{1}+H_{2}}(x)$ of the limit law $H_{1} + H_{2}$, supported on $[0,\infty)$, is given in terms of the density function $\varphi(x)=\frac1{\sqrt{2\pi}}e^{-x^2/2}$ and the cumulative distribution function $\Phi(x)=\int_{-\infty}^x \varphi(t)dt$ of the standard normal distribution as follows:
\[
f(x)=4\varphi(x)\cdot\big(2\Phi(x)-1\big).
\]
\end{theorem}

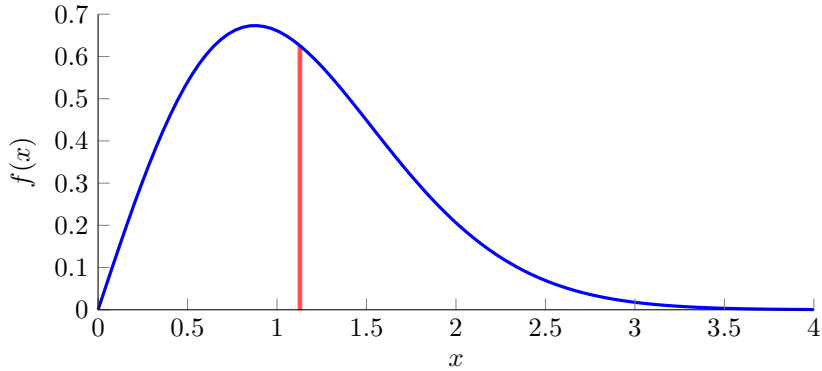
\begin{figure}[!htb]
\begin{tikzpicture}
\begin{axis}[
 domain=0:4,samples=100,%smooth,  enlargelimits=upper, 
  axis lines*=left,
  height=5.5cm, width=11cm,
  xtick={0,0.5, 1,1.5, 2,2.5, 3,3.5, 4}, 
	 ytick={0,0.1, 0.2, 0.3, 0.4, 0.5, 0.6, 0.7}, 
	ymin=0, ymax=0.7,
	xmin=0, xmax=4,
   axis on top,  clip=false, no markers,
	 xlabel={$x$},
   ylabel={$f(x)$},
  ]
%\addplot[very thick,cyan!50!black]{gauss(2,1)};
\addplot[red!70, fill=red!70, domain=1.118:1.138]{0.62548}\closedcycle; 
\addplot [very thick, blue] {(4/sqrt(2*pi))*exp(-x^2/2)*(2*normcdf(x)-1)};
	%\addplot[domain=0:2,very thick]{(x^2)/2+2)};
\end{axis}
\end{tikzpicture}
\caption{Plot of the density function $f(x)$ of the limit law.
The red vertical line marks the expected value $\frac{2}{\sqrt{\pi}}\approx 1.12838$ 
and the corresponding value of the density $f(\frac{2}{\sqrt{\pi}})\approx 0.62548$.}
\label{fig:LimitDensity}%
\end{figure}

The shape of the density matches very well the simulations in~\cite[Figure 3]{KSTY2023}, as well as our own simulations shown before.

\begin{proof}
Our starting point is the distributional equation of $X_{n}$ stated in Proposition~\ref{prop1}. 
In the following we give a derivation of the limit law for the first summand $Y_{h-J_h,J_h}$, the second one is analyzed in a similar way. 
Let $x>0$. We use the de Moivre-Laplace asymptotics of the binomial distribution and get, with $\mu_{J}=h/2$ and $\sigma_{J} = \sqrt{h}/2$ as denoted in \eqref{eqn:binomial_approximation},
\begin{align*}
F_h(x) & := \P\big\{Y_{h-J_h,J_h}\le x\sqrt{h/2}\big\} \\
& \sim \int_0^{h} \P\big\{Y_{h-j,j}\le x\sqrt{h/2}\big\} \exp\Big(-\frac{(j-\mu_J)^2}{2\sigma_J^2}\Big)\frac{1}{\sigma_J\sqrt{2\pi}}dj.
\end{align*}
%This implies that
%\begin{align*}
%F_h(x)\sim 
%\int_0^{h} \P\{Y_{h-j,j}\le x\sqrt{h/2}\} \exp\Big(-\frac{(j-\mu_J)^2}{2\sigma_J^2}\Big)\frac{1}{\sigma_J\sqrt{2\pi}}dj
%\end{align*}
By the asymptotics of the binomial random variable we further get by substituting $j=\mu_{J}+\sigma_{J} t$: 
\begin{align*}
F_h(x)\sim 2\int_{0}^{\infty} \P\big\{Y_{h/2-t\sqrt{h}/2,h/2+t\sqrt{h}/2}\le x\sqrt{h/2}\big\} \exp\Big(-\frac{t^2}{2}\Big)\frac{1}{\sqrt{2\pi}}dt.
\end{align*}
Furthermore, by our previous result in Proposition~\ref{prop:LinExp} we obtain
\[
\P\{Y_{h/2-t\sqrt{h}/2,h/2+t\sqrt{h}/2}\le x\sqrt{h/2}\}\to 1-e^{-\frac{x(2\sqrt{2}t +x)}{4}},
\]
as 
\[
t\sqrt{h}\sim t\sqrt{2}\cdot \sqrt{h/2}.
\]
This implies that
\[
F_h(x) \sim 1 - \frac{2}{\sqrt{2\pi}}\int_{0}^{\infty}e^{-\frac{x(2\sqrt{2}t +x)}{4}-t^2/2}dt 
= 1 - \frac{2}{\sqrt{2\pi}} \int_{0}^{\infty} e^{-\big(\frac{\sqrt{2}t+x}{2}\big)^{2}} dt.
\]
The integral is readily evaluated by substituting $\tau = (\sqrt{2}t+x)/2$ and we get
\[
F_h(x)\sim 1-\frac{2}{\sqrt{\pi}}\int_{x/2}^\infty e^{-\tau^2}d\tau.
\]
This implies that the arising density of the limit law of $Y_{h-J_{h},J_{h}}/\sqrt{h/2}$ is obtained by taking the derivative w.r.t.\ $x$ yielding
\[
\frac{e^{-x^2/4}}{\sqrt{\pi}}, \quad x \ge 0.
\]
Actually, this is the density of a half-normal distributed random variable $\text{HN}(c)$ with 
parameter $c=\sqrt{2}$. Finally, we note that 
\[
\sqrt{n}\sim \sqrt{4\cdot \frac{h}2}=2\sqrt{\frac{h}2},
\]
which implies that the limit law of $Y_{h-J_h,J_h}/\sqrt{n}$ has, due to scaling, the density 
\[
f_{H}(x)=\frac2{\sqrt{\pi}}e^{-x^2},\quad x\ge 0.
\]
Finally, the density of the limit law $H_{1}+H_{2}$ of $X_{n}/\sqrt{n}$ can be obtained using standard methods, where $\erf(x) = \frac{2}{\sqrt{\pi}}\int_{0}^{x} e^{-t^{2}} dt$ denotes the error function:
\begin{align*}
f_{H_1+H_2}(x)&=\int_0^{x}f_{H_1}(t)f_{H_2}(x-t)dt
=\frac{4}{\pi}\int_0^{x}e^{-t^2-(x-t)^2}dt
=\frac{4}{\pi}e^{-x^2}\int_0^{x}e^{-2t(t-x)}dt\\
& = \frac{4}{\pi} e^{-x^{2}/2} \int_{0}^{x} e^{-2(t-\frac{x}{2})^{2}} dt =\frac{2\sqrt{2}}{\sqrt{\pi}}e^{-x^2/2}\erf\big(\frac{x}{\sqrt{2}}\big)=4\varphi(x)\cdot\big(2\Phi(x)-1\big).
\end{align*}

\end{proof}

\section{Convergence of moments}
We start by collecting the moment sequence of the limit law. Then, we perform a sanity check by comparing the moments of the limit law 
$H_1 + H_2$ of $X_n/\sqrt{n}$ with the precise results for the first five moments with $n=4L$ tending to infinity, computed in~\cite{KSTY2023}. 
Afterwards, we extend the results of~\cite{KSTY2023} obtaining moment convergence for arbitrary high integer moments (for arbitrary $n\to\infty$).

\smallskip

\begin{lem}[Moments of the limit law] 
The sum  $H_1 + H_2$ of two independent identically distributed half-normal distributed random variables $H_1$, $H_2$, each one
with density function $f_{H}(x) = \frac{2}{\sqrt{\pi}}e^{-x^2}$, $x \ge 0$, has moment sequence
\[
\tilde{\mu}_{s} %= \E\Big(\big(H_1+H_2\big)^s\Big)% = \sum_{k=0}^{s}\binom{s}k \mu_k\mu_{s-k}
= \sum_{k=0}^{s}\binom{s}k \frac{\Gamma(\frac{k+1}2)\Gamma(\frac{s-k+1}2)}{\Gamma^2(\frac{1}2)},
\]
for integer $s\ge 1$.
\end{lem}
\begin{proof}
We recall basic properties of the half-normal distribution $H$ with density
$f_{H}(x) = \frac{2}{\sqrt{\pi}}e^{-x^2}$. Its integer moments are given by
\[
\mu_s=\E(H^s) = \frac{\Gamma(\frac{s+1}2)}{\Gamma(\frac12)} = \frac{\Gamma(\frac{s+1}2)}{\sqrt{\pi}},
\]
with $s\ge 0$. Consequently, the $s$-th integer moments $\tilde{\mu}_{s}$ of $H_1 + H_2$ are given by
\begin{equation}\label{eqn:moments_H1+H2}
\tilde{\mu}_{s} = \E\Big(\big(H_1+H_2\big)^s\Big) = \sum_{k=0}^{s}\binom{s}k \mu_k\mu_{s-k}
= \sum_{k=0}^{s}\binom{s}k \frac{\Gamma(\frac{k+1}2)\Gamma(\frac{s-k+1}2)}{\Gamma^2(\frac{1}2)}. 
\end{equation}
\end{proof}

\begin{example}[Asymptotics of the first five moments~\cite{KSTY2023}]
The first five moments of the half normal distributed random variable are given by
\[
\mu_1=\mu_{3}=\frac{1}{\sqrt{\pi}},\ 
\mu_2=\frac12,\ \mu_4=\frac34,\
\mu_5=\frac{2}{\sqrt{\pi}}.
\]
This implies that the first five moments of the limit law are given by
\begin{equation}
\label{eqn:mom1}
\tilde{\mu}_{1}=\frac{2}{\sqrt{\pi}},\ 
\tilde{\mu}_{2}=\frac{2}{\sqrt{\pi}}+1,\ \tilde{\mu}_{3}=\frac5{\sqrt{\pi}},\
\tilde{\mu}_{4}=\frac8\pi +3,\ \tilde{\mu}_{5}=\frac{43}{2\sqrt{\pi}}.
\end{equation}
We check with the asymptotic expansions of the moments
$\E(X_n^s)$, $1\le s\le 5$ for $n=4L$, and summarize here for the reader's convenience the dominant terms of the first five integer moments of $X_{n}$, with $n=4L$, as stated in~\cite{KSTY2023}:
\begin{align*}
\E(X_n)&\sim \frac{4L}{4^L}\binom{2L}{L},\\
\E(X_n^2)&\sim \frac{(4L)^2}{2 \cdot 4^{2L}}\binom{2L}{L}^2+4L,\\
\E(X_n^3)&\sim \frac{40L^2}{ 4^{L}}\binom{2L}{L},\\
\E(X_n^4)&\sim \frac{256L^3}{2\cdot 4^{2L}}\binom{2L}{L}^2+48L^2,\\
\E(X_n^5)&\sim \frac{688L^3}{4^{L}}\binom{2L}{L}.
\end{align*}
By the standard asymptotics 
\begin{equation}
\label{eqn:mom2}
\frac{2L}{4^L}\binom{2L}{L}\sim \frac{\sqrt{n}}{\sqrt{\pi}}
\end{equation}
and our results for the first five moments~\eqref{eqn:mom1}, we observe that indeed, for $n=4L$, it holds
\[
\E\big(\big(X_n/\sqrt{n}\big)^s\big)\to \tilde{\mu}_{s},\quad 1\le s\le 5.
\]
\end{example}

\smallskip

The remaining part of the section is devoted to prove that actually all integer moments of $X_{n}/\sqrt{n}$ converge to the moments of $H_{1}+H_{2}$. Again we deal with the generating functions description given in Lemma~\ref{lem:no1}, but we proceed in treating the occurring recurrence for $g_{m_{1},m_{2}}(q)$ in a direct way by using the so-called kernel method, see~\cite{BanderierFlajolet2002,Prodinger2004}. To this aim we first rewrite \eqref{eqn:no1} by setting $M:=m_{1}+m_{2}$ and $d:=m_{2}-m_{1}$, and introducing
\begin{equation*}
  \tilde{g}_{M,d}(q) = g_{m_{1},m_{2}}(q)=g_{(M-d)/2,(M+d)/2}(q).
\end{equation*}
This yields the recurrence
\begin{align}
  & \tilde{g}_{M,d}(q) = q^{\delta(d,-1)+\delta(d,-2)} \cdot \tilde{g}_{M-1,d+1} + \tilde{g}_{M-1,d-1}, \quad \text{for $M \ge 1$ and $|d| \le M$},\notag\\
	& \tilde{g}_{0,0}(q) =1, \qquad \text{with} \quad \tilde{g}_{M,d}=0, \quad \text{for $M<0$ or $|d|>M$}.
	\label{eqn:rec_gMd}
\end{align}
Let us introduce the trivariate generating function
\begin{equation}
  \tilde{G}(z,u,q) = \sum_{M \ge 0} \sum_{d=-M}^{M} \tilde{g}_{M,d}(q) z^{M} u^{d} = \sum_{m_{1},m_{2} \ge 0} g_{m_{1},m_{2}}(q) z^{m_{1}+m_{2}} u^{m_{2}-m_{1}},
\end{equation}
for which we get an explicit solution.
\begin{prop}\label{prop:formula_Gzuq}
  The generating function $\tilde{G}(z,u,q)$ of the sequence $\tilde{g}_{M,d}(q)$ (and thus also $g_{m_{1},m_{2}}(q)$, resp.) satisfying recurrence~\eqref{eqn:rec_gMd} (and \eqref{eqn:no1}, resp.) is given by
	\begin{equation}
	  \tilde{G}(z,u,q) = \frac{\big(2qu^{2}z+(q-1)^{2}u-q(q-1)z\big)P(z^{2})-zu(u+q(q-1)z)}{zu(zu^{2}-u+z)(1-2qP(z^{2}))},
	\end{equation}
	where $P(t) = \frac{1-\sqrt{1-4t}}{2} = \sum_{n \ge 1} \frac{1}{n} \binom{2n-2}{n-1} t^{n}$ denotes the generating function of shifted Catalan-numbers.
\end{prop}
\begin{proof}
  Introducing the auxiliary functions $\tilde{G}_{d}(z,q) = \sum_{M \ge 0} \tilde{g}_{M,d}(q) z^{M}$, for $d \in \mathbb{Z}$, we immediately obtain from \eqref{eqn:rec_gMd} the system of equations
	\begin{align}
	  \tilde{G}_{d}(z,q) & = z \tilde{G}_{d+1}(z,q) + z \tilde{G}_{d-1}(z,q), \quad \text{$d \ge 1$ or $d \le -3$},\notag\\
		\tilde{G}_{0}(z,q) & = z \tilde{G}_{1}(z,q) + z \tilde{G}_{-1}(z,q) + 1,\label{eqn:system_Gdzq}\\
		\tilde{G}_{d}(z,q) & = qz \tilde{G}_{d+1}(z,q) + z \tilde{G}_{d-1}(z,q), \quad \text{$d =-1$ or $d =-2$}.\notag
	\end{align}
	In order to treat \eqref{eqn:system_Gdzq} by means of the kernel method, we introduce the pair of functions
	\begin{equation*}
	  A(z,u,q) = \sum_{d \ge 0} \tilde{G}_{-d}(z,q) u^{d}, \qquad B(z,u,q) = \sum_{d \ge 0} \tilde{G}_{d}(z,q) u^{d}.
	\end{equation*}
	It is straightforward to get from \eqref{eqn:system_Gdzq} the following pair of equations for $A(z,u,q)$ and $B(z,u,q)$, which involve the unknown functions $\tilde{G}_{0}(z,q)$ and $\tilde{G}_{-1}(z,q)$:
	\begin{equation}\label{eqn:system_Azuq_Bzuq}
	\begin{split}
	  & (zu^{2}-u+z) A(z,u,q)\\
		& \qquad \qquad \mbox{} + ((q-1)zu^{2}+u-z)\tilde{G}_{0}(z,q) +((q-1)zu^{2}-zu)\tilde{G}_{-1}(z,q) = 0,\\
		& (zu^{2}-u+z) B(z,u,q) -z\tilde{G}_{0}(z,q) +zu\tilde{G}_{-1}(z,q)+u = 0.
	\end{split}
	\end{equation}
	The roots $u_{1}=u_{1}(z)$, $u_{2}=u_{2}(z)$ satisfying $K(z,u) := zu^{2}-u+z = z(u-u_{1})(u-u_{2}) = 0$, i.e., vanishing the kernel $K(z,u)$, are the ones stated in \eqref{eqn:u1_u2}. Plugging the root $u_{1}$, which admits a power series expansion around $z=0$, for $u$ into \eqref{eqn:system_Azuq_Bzuq}, the kernel is annihilated leading to a linear system of equations for $\tilde{G}_{0}(z,q)$ and $\tilde{G}_{-1}(z,q)$, whose solution can be written in the following way (note that $P(z^{2})=z u_{1}(z)$):
\begin{equation}\label{eqn:sol_G0zq_G1zq}
  \tilde{G}_{0}(z,q) = \frac{(1-q)P(z^{2})+qz^{2}}{z^{2}(1-2qP(z^{2}))}, \qquad \tilde{G}_{-1}(z) = \frac{qP(z^{2})}{z(1-2qP(z^{2}))}.
\end{equation}	
According to \eqref{eqn:system_Azuq_Bzuq} and \eqref{eqn:sol_G0zq_G1zq}, we thus also get explicit solutions for the auxiliary trivariate g.f.:
\begin{align*}
  A(z,u,q) & = \frac{(z-u-(q-1)zu^{2})\tilde{G}_{0}(z,q)+(zu-(q-1)zu^{3})\tilde{G}_{-1}(z,q)}{zu^{2}-u+z},\\
	B(z,u,q) & = \frac{z\tilde{G}_{0}(z,q)-zu\tilde{G}_{-1}(z,q)-u}{zu^{2}-u+z}.
\end{align*}
Finally, according to the definition, we have
\begin{equation*}
  \tilde{G}(z,u,q) = A(z,u^{-1},q) + B(z,u,q) - \tilde{G}_{0}(z,q),
\end{equation*}
which, after simple manipulations, yields the stated result.
\end{proof}
Bearing in mind the representation of $f_{n}(q)$ given in Lemma~\ref{lem:no1}, we will set $u=1$, i.e., considering
\begin{equation}\label{eqn:Def_TildeGzq}
  \tilde{G}(z,q) := \tilde{G}(z,1,q) = \sum_{M \ge 0} \sum_{d=-M}^{M} \tilde{g}_{M,d}(q) z^{M} = \sum_{M \ge 0} \sum_{m=0}^{M} g_{m,M-m}(q) z^{M}.
\end{equation}
Namely, when setting $R(q) = 4q^{4}-2(q^{2}+q^{3})$, this yields
\begin{equation}\label{eqn:rewritten_fnq}
  f_{n}(q) = R(q) + [z^{h}] \tilde{G}(z,q) \cdot [z^{n-h}] \tilde{G}(z,q).
\end{equation}
As we are interested in the moments of $X_{n}$, we will set $q=1+w$ in \eqref{eqn:rewritten_fnq} and carry out a series expansion around $w=0$. According to the definition of $f_{n}(q)$, the coefficients in the corresponding expansion involve the factorial moments of $X_{n}$,
\begin{equation*}
  f_{n}(1+w) = \sum_{s \ge 0} \frac{2^{n} \mathbb{E}(X_{n}^{\underline{s}})}{s!} w^{s}.
\end{equation*}
When denoting the respective expansions of the remaining functions via
\begin{equation*}
  R(1+w) = \sum_{0 \le s \le 4} r_{s} w^{s}, \qquad \tilde{G}(z,1+w) = \sum_{s \ge 0} \tilde{g}_{s}(z) w^{s},
\end{equation*}
we thus get from \eqref{eqn:rewritten_fnq}, by extracting the coefficient of $w^{s}$, the useful representation:
\begin{equation}\label{eqn:representation_moments_Xn}
  \mathbb{E}(X_{n}^{\underline{s}}) = \frac{s! \, r_{s}}{2^{n}} + \sum_{k=0}^{s} \frac{s!}{2^{n}} \sum_{k=0}^{s} \big([z^{h}] \tilde{g}_{k}(z)\big) \cdot \big([z^{n-h}] \tilde{g}_{s-k}(z)\big).
\end{equation}
Of course, to make use of it, we require suitable expansions of $\tilde{G}(z,1+w)$. First we use Proposition~\ref{prop:formula_Gzuq} and set $u=1$ to get the explicit formula (with $P(t)$ defined above):
\begin{equation}\label{eqn:formula_Gzq}
  \tilde{G}(z,q) = \frac{1}{1-2z} + \frac{(q-1)\big(qz^{2}+(1-q+qz)P(z^{2})\big)}{z(1-2z)(1-2qP(z^{2}))}.
\end{equation}
Next we consider the coefficients in the series expansion of $\tilde{G}(z,1+w)$ around $w=0$, i.e., the functions $\tilde{g}_{s}(z) = [w^{s}] \tilde{G}(z,1+w)$, determine the dominant singularities and provide local expansions around the singularity yielding the main asymptotic contributions.
\begin{lem}\label{lem:TildeGzq_Expansion}
  Let $\tilde{G}(z,q)$ as defined in \eqref{eqn:Def_TildeGzq}. The functions $\tilde{g}_{s}(z) = [w^{s}] \tilde{G}(z,1+w)$ obtained as coefficients in a series expansion of $\tilde{G}(z,1+w)$ around $w=0$ have radius of convergence $\frac{1}{2}$ and, for $s \ge 1$, have the two dominant singularities $\rho_{1,2} = \pm \frac{1}{2}$. Moreover, the local behaviour of $\tilde{g}_{s}(z)$ around $\rho := \rho_{1} = \frac{1}{2}$ is given as follows,
\begin{equation*}
  \tilde{g}_{s}(z) = \frac{1}{2^{\frac{s}{2}} \, (1-2z)^{\frac{s}{2}+1}} \cdot \big(1+\mathcal{O}(\sqrt{1-2z})\big), \quad s \ge 0.
\end{equation*}
\end{lem}
\begin{remark}\label{rem:Second_Singularity}
  It is not difficult to show that the second dominant singularity $\rho_{2} = -\frac{1}{2}$ occurring in the functions $\tilde{g}_{s}(z)$ of Lemma~\ref{lem:TildeGzq_Expansion} leads to contributions that do not affect the main terms stemming from the contributions of the singularity $\rho = \rho_{1} = \frac{1}{2}$. Since we are here only interested in the main term contribution, we will restrict ourselves to elaborate the expansion around $\rho$.
\end{remark}
\begin{proof}
The explicit formula for $\tilde{G}(z,q)$ given in \eqref{eqn:formula_Gzq} can be rewritten as
\begin{equation*}
  \tilde{G}(z,q) = \frac{1}{1-2z} + \frac{(q-1)\big(z^{2}+zP(z^{2}) + (q-1)(z^{2}+(z-1)P(z^{2}))\big)}{z(1-2z)(1-2P(z^{2}))\big(1-\frac{2(q-1)P(z^{2})}{1-2P(z^{2})}\big)}.
\end{equation*}
Setting $q=1+w$, we obtain the series expansion
\begin{align*}
  & \tilde{G}(z,1+w) = \frac{1}{1-2z} + \frac{w\big(z^{2}+zP(z^{2}) + w(z^{2}+(z-1)P(z^{2}))\big)}{z(1-2z)(1-2P(z^{2}))\big(1-\frac{2wP(z^{2})}{1-2P(z^{2})}\big)}\\
	&  \quad = \frac{1}{1-2z} + w \frac{P(z^{2})+z}{(1-2z)(1-2P(z^{2}))}\\
	& \qquad \mbox{} + \sum_{s \ge 2} w^{s} \frac{(z+P(z^{2}))(2P(z^{2}))^{s-1}}{(1-2z)(1-2P(z^{2}))^{s}}
	+ \sum_{s \ge 2} w^{s} \frac{(z^{2}+(z-1)P(z^{2}))(2P(z^{2}))^{s-2}}{z(1-2z)(1-2P(z^{2}))^{s-1}}\\
	& = \frac{1}{1-2z} + w \frac{P(z^{2})+z}{(1-2z)(1-2P(z^{2}))}
	+ \sum_{s \ge 2} w^{s} \frac{(2P(z^{2}))^{s-2} ((z+1)P(z^{2})-z^{2})}{z(1-2z)(1-2P(z^{2}))^{s}},
\end{align*}
where we used in the last step the relation $(P(z^{2}))^{2} = P(z^{2}) - z^{2}$.
Thus, the functions $\tilde{g}_{s}(z)$ are given as follows:
\begin{gather}
  \tilde{g}_{0}(z) = \frac{1}{1-2z}, \qquad \tilde{g}_{1}(z) = \frac{P(z^{2})+z}{(1-2z)(1-2P(z^{2}))},\notag\\
	\tilde{g}_{s}(z) = \frac{\big((z+1)P(z^{2})-z^{2}\big) (2P(z^{2}))^{s-2}}{z(1-2z)(1-2P(z^{2}))^{s}}, \quad s \ge 2.\label{eqn:formulae_gsz}
\end{gather}
Since $1-2P(z^{2}) = \sqrt{1-4z^{2}} = \sqrt{1-2z} \cdot \sqrt{1+2z}$, it is immediate from the explicit formul\ae, that the functions $\tilde{g}_{s}(z)$ have radius of convergence $\frac{1}{2}$ with dominant singularities $\rho_{1}=\frac{1}{2}$ and, for $s \ge 1$, $\rho_{2}=-\frac{1}{2}$. To describe the local behaviour of $\tilde{g}_{s}(z)$ around $\rho:=\rho_{1}$ we use the notation $\mathcal{Z} = \frac{1}{1-2z}$ and $\tilde{\mathcal{Z}} = \frac{1}{1-4z^{2}} = \frac{1}{(1-2z)(1+2z)}$. We collect a few local expansions around $\rho$ used thereafter:
\begin{gather*}
  z = \frac{1}{2} \cdot \big(1+\mathcal{O}(\mathcal{Z}^{-1})\big), \qquad \tilde{Z} = \frac{\mathcal{Z}}{2} \cdot \big(1+\mathcal{O}(\mathcal{Z}^{-1})\big),\\
	P(z^{2}) = \frac{1}{2} \cdot \big(1+\mathcal{O}(\mathcal{Z}^{-\frac{1}{2}})\big), \qquad 
	\frac{1}{1-2P(z^{2})} = \frac{\mathcal{Z}^{\frac{1}{2}}}{2^{\frac{1}{2}}} \cdot \big(1+\mathcal{O}(\mathcal{Z}^{-1})\big).
\end{gather*}
We then immediately get from \eqref{eqn:formulae_gsz}:
\begin{equation*}
  \tilde{g}_{0}(z) = \mathcal{Z}, \qquad \tilde{g}_{1}(z) = \frac{\mathcal{Z}^{\frac{3}{2}}}{2^{\frac{1}{2}}} \cdot \big(1+\mathcal{O}(\mathcal{Z}^{-\frac{1}{2}})\big).
\end{equation*}
Furthermore, with these expansions we easily obtain, for $s \ge 2$ arbitrary but fixed,
\begin{align*}
  \tilde{g}_{s}(z) & = \big(1+\mathcal{O}(\mathcal{Z}^{-\frac{1}{2}})\big) \cdot \big(\frac{1}{2} + \mathcal{O}(\mathcal{Z}^{-\frac{1}{2}})\big) \cdot 2 \big(1+\mathcal{O}(\mathcal{Z}^{-1})\big) \cdot \mathcal{Z} \cdot \frac{\mathcal{Z}^{\frac{s}{2}}}{2^{\frac{s}{2}}} \cdot \big(1+\mathcal{O}(\mathcal{Z}^{-1})\big)\\
	& = \frac{\mathcal{Z}^{\frac{s}{2}+1}}{2^{\frac{s}{2}}} \cdot \big(1+\mathcal{O}(\mathcal{Z}^{-\frac{1}{2}})\big),
\end{align*}
which completes the proof.
\end{proof}
Now we have all ingredients at hand to show convergence of the moments of $X_{n}$.
\begin{theorem}\label{the:Xn_Moments}
The $s$-th integer moments $\E\big(\big(\frac{X_{n}}{\sqrt{n}}\big)^{s}\big)$ of the suitably scaled r.v.\ $X_{n}$ converge, for arbitrary but fixed $s \ge 0$ and $n \to \infty$, to the moments of the limit law $H_{1}+H_{2}$:
\begin{equation*}
\E\Big(\big(\frac{X_{n}}{\sqrt{n}}\big)^{s}\Big) \to \tilde{\mu}_{s} = \E\big(\big(H_{1}+H_{2}\big)^{s}\big) = \sum_{k=0}^{s} \binom{s}{k} \cdot \frac{\Gamma\big(\frac{k+1}{2}\big) \Gamma\big(\frac{s-k+1}{2}\big)}{\pi}.
\end{equation*}
\end{theorem}

\begin{proof}[Proof of Theorem~\ref{the:Xn_Moments}]
We consider the representation~\eqref{eqn:representation_moments_Xn} of the $s$-th factorial moments of $X_{n}$ and first observe that $r_{s}$, the contributions of $R(q)$, are bounded, $|r_{s}| \le 24$, thus turn out to be exponentially small compared to the remaining contributions. Consequently, they can be neglected, yielding
\begin{equation}\label{eqn:facmom_Xn_reduced}
  \mathbb{E}(X_{n}^{\underline{s}}) \sim \sum_{k=0}^{s} \frac{s!}{2^{n}} \sum_{k=0}^{s} \big([z^{h}] \tilde{g}_{k}(z)\big) \cdot \big([z^{n-h}] \tilde{g}_{s-k}(z)\big).
\end{equation}
In order to extract coefficients from the functions $\tilde{g}_{s}(z)$ asymptotically, we use Lemma~\ref{lem:TildeGzq_Expansion} and apply transfer lemmata \cite{FlaSed} that ``translate'' the local behaviour of the generating function near the dominant singularity to the asymptotic behaviour of their coefficients. The local expansion around $\rho=\frac{1}{2}$ given there (see also Remark~\ref{rem:Second_Singularity}) immediately leads to the following asymptotic behaviour, for $n \to \infty$ and arbitrary but fixed $s \ge 0$:
\begin{equation*}
  [z^{n}] \tilde{g}_{s}(z) = \frac{2^{n} n^{\frac{s}{2}}}{2^{\frac{s}{2}} \Gamma(\frac{s}{2}+1)} \cdot \big(1+\mathcal{O}(n^{-\frac{1}{2}})\big) = \frac{2^{n} 2^{\frac{s}{2}} \Gamma(\frac{s+1}{2})}{s! \sqrt{\pi}} \cdot \big(1+\mathcal{O}(n^{-\frac{1}{2}})\big),
\end{equation*}
where we used for the latter equation the duplication formula of the Gamma function, $\Gamma(\frac{s}{2}+1) \Gamma(\frac{s+3}{2}) = \sqrt{\pi} \, 2^{-s-1} \Gamma(s+2)$, and $\Gamma(\frac{1}{2})=\sqrt{\pi}$.

Plugging this asymptotic result into \eqref{eqn:facmom_Xn_reduced} and using $h=\frac{n}{2} \cdot \big(1+\mathcal{O}(n^{-1})\big)$, we get
\begin{equation}
\begin{split}
\label{eqn:final}
  \mathbb{E}(X_{n}^{\underline{s}}) & = \sum_{k=0}^{s} \frac{s!}{2^{n}} \cdot \frac{2^{h} 2^{\frac{k}{2}} \Gamma(\frac{k+1}{2}) h^{\frac{k}{2}}}{\sqrt{\pi} \, k!} \cdot \big(1+\mathcal{O}(h^{-\frac{1}{2}})\big)\\
	& \qquad \cdot \frac{2^{n-h} 2^{\frac{s-k}{2}} \Gamma(\frac{s-k+1}{2}) (n-h)^{\frac{s-k}{2}}}{\sqrt{\pi} \, (s-k)!} \cdot \big(1+\mathcal{O}((n-h)^{-\frac{1}{2}})\big)\\
	& = \sum_{k=0}^{s} \binom{s}{k} \cdot \frac{\Gamma(\frac{k+1}{2}) \Gamma(\frac{s-k+1}{2})}{\pi} \cdot n^{\frac{s}{2}} \cdot \big(1+\mathcal{O}(n^{-\frac{1}{2}})\big).
\end{split}
\end{equation}
The raw moments can be expressed in terms of the factorial moments by
\[
\E(X_{n}^{s}) = \sum_{k=0}^{s}\Stir{s}k\E(X_{n}^{\underline{k}}),
\]
where $\Stir{s}{k}$ denote the Stirling numbers of the second kind, counting the number
of ways to partition a set of $s$ objects into $k$ non-empty subsets.
From~\eqref{eqn:final} we thus obtain 
\[
\E(X_{n}^{s}) = \E(X_{n}^{\underline{s}}) + \mathcal{O}\big(\E(X_{n}^{\underline{s-1}})\big)
\]
for $n \to \infty$ and arbitrary but fixed $s \ge 1$. 
This leads to the expansion
\begin{equation*}
  \mathbb{E}(X_{n}^{s}) = \sum_{k=0}^{s} \binom{s}{k} \cdot \frac{\Gamma(\frac{k+1}{2}) \Gamma(\frac{s-k+1}{2})}{\pi} \cdot n^{\frac{s}{2}} \cdot \big(1+\mathcal{O}(n^{-\frac{1}{2}})\big) = \tilde{\mu}_{s} \cdot n^{\frac{s}{2}} \cdot \big(1+\mathcal{O}(n^{-\frac{1}{2}})\big).
\end{equation*}
Scaling of $X_n$ by $\sqrt{n}$ immediately yields the stated result.
\end{proof}

\begin{remark}[Proof of Theorem~\ref{the:limitLaw} by the method of moments]
Finally, we note that Theorem~\ref{the:Xn_Moments} also strengthens Theorem~\ref{the:limitLaw}. 
Carleman's criterion~\cite[pp.~189--220]{Carleman23} for the Stieltjes moment problem, support $[0,\infty)$, states
that if 
\begin{equation}
\sum_{s=0}^{\infty}\mu_s^{-1/(2s)}=+\infty,
\label{eq:Carleman}
\end{equation}
then the moment sequence $(\mu_s)_{s\ge 1}$ determines a unique distribution. Furthermore, this implies 
that if there exists a constant $C>0$ such that
\[
\mu_s\le C^s(2s)!\quad\text{for } s\in\N,
\]
then Carleman's criterion is satisfied. We note that for the sum of independent half-normals $H_1+H_2$ with
moment sequence $(\tilde{\mu}_{s})_{s\in\N}$ it holds
\begin{equation*}
\begin{split}
\tilde{\mu}_{s} 
&= \sum_{k=0}^{s}\binom{s}k \frac{\Gamma(\frac{k+1}2)\Gamma(\frac{s-k+1}2)}{\Gamma^2(\frac{1}2)}
\le \frac1\pi\sum_{k=0}^{s}\binom{s}k \Gamma(k+1)\Gamma(s-k+1)
=\frac{s}\pi s! \le (2s)!,
\end{split}
\end{equation*}
such that the divergence in Carleman's criterion is satisfied. Thus, by the Fr\'echet--Shohat theorem~\cite{FrSh1931}, 
we obtain the weak convergence of the normalized random variable $\frac{X_n}{\sqrt{n}}$ to $H_1+H_2$ with moment sequence $(\tilde{\mu}_s)_{s\in\N}$. 
\end{remark}

\section{Conclusion}
We studied the number of correct guesses when starting with an ordered deck of
$n$ cards labeled $1$ up to $n$ is riffle-shuffled exactly one time. Assuming that no feedback is given to the person guessing, 
the limit law was determined. Additionally, we have shown convergence of all positive integer moments, providing
a second proof of the limit law. We note that the approach of this work also allows to analyze
different questions, like waiting times for correct guesses, etc.

\section*{Declarations of interest}
The authors declare that they have no competing financial or personal interests that influenced the work reported in this paper. 

\bibliographystyle{cyrbiburl}
\bibliography{CardGuessingNoOne-refs}{}

\begin{thebibliography}{10}

\bibitem{AldousDiaconis1986}
D.~Aldous and P.~Diaconis.
\newblock Shuffling cards and stopping times.
\newblock  \textit{Amer. Math. Monthly}, 93:333--348, 1986.

\bibitem{BanderierFlajolet2002}
C.~Banderier and P.~Flajolet.
\newblock Basic analytic combinatorics of directed lattice paths.
\newblock  \textit{Theoretical Computer Science}, 281(1--2):37--80, 2002.

\bibitem{BFSS2001}
C.~Banderier, P.~Flajolet, G.~Schaeffer, and M.~Soria.
\newblock \href{https://lipn.univ-paris13.fr/~banderier/Papers/rsa.pdf}{Random
  maps, coalescing saddles, singularity analysis, and {A}iry phenomena}.
\newblock  \textit{{Random Struct. Algorithms}}, 19(3-4):194--246, 2001.

\bibitem{BKW2024}
C.~Banderier, M.~Kuba, and M.~Wallner.
\newblock Phase transitions of composition schemes: Mittag-leffler and mixed
  poisson distributions.
\newblock Accepted for publication in Annals of Applied Probability.

\bibitem{BlackwellHodges1957}
D.~Blackwell and J.~L.~Hodges Jr.
\newblock Design for the control of selection bias.
\newblock  \textit{The Annals of Mathematical Statistics}, 28(2):449--460,
  1957.

\bibitem{Carleman23}
T.~Carleman.
\newblock  \textit{Sur les \'equations int\'egrales singuli\`eres \`a noyau
  r\'eel et sym\'etrique}.
\newblock Uppsala Universitets {\AA}rsskrift, 1923.

\bibitem{Ciucu1998}
M.~Ciucu.
\newblock No-feedback card guessing for dovetail shuffles.
\newblock  \textit{Ann. Appl. Probab.}, 8(4):1251--1269, 1998.

\bibitem{Diaconis1978}
P.~Diaconis.
\newblock Statistical problems in esp research.
\newblock  \textit{Science}, 201(4351):131--136, 1978.

\bibitem{DiaconisGraham1981}
P.~Diaconis and R.~Graham.
\newblock The analysis of sequential experiments with feedback to subjects.
\newblock  \textit{Annals of Statistics}, 9(1):3--23, 1981.

\bibitem{DiaconisMcPitman1995}
P.~Diaconis, M.~McGrath, and J.~Pitman.
\newblock Riffle shuffles, cycles, and descents.
\newblock  \textit{Combinatorica}, 15:11--29, 1995.

\bibitem{Efron1971}
B.~Efron.
\newblock Forcing a sequential experiment to be balanced.
\newblock  \textit{Biometrika}, 58(3):403--417, 1971.

\bibitem{Fisher1936}
R.~A. Fisher.
\newblock Design of experiments.
\newblock  \textit{British Medical Journal}, 1(3923):554, 1936.

\bibitem{FlaSed}
P.~Flajolet and R.~Sedgewick.
\newblock  \textit{Analytic Combinatorics}.
\newblock Cambridge University Press, 2009.

\bibitem{FrSh1931}
M.~Fr\'echet and J.~Shohat.
\newblock \href{https://www.jstor.org/stable/1989421?seq=1}{A proof of the
  generalized second limit theorem in the theory of probability}.
\newblock  \textit{Trans. Amer. Math. Soc.}, 33(2):533--543, 1931.

\bibitem{Gilbert1955}
E.~Gilbert.
\newblock Theory of shuffling.
\newblock Technical memorandum, Bell Labs, 1955.

\bibitem{HeOttolini2021}
J.~He and A.~Ottolini.
\newblock \href{https://arxiv.org/abs/2108.07355}{Card guessing and the
  birthday problem for sampling without replacement}.
\newblock Accepted for publication in Annals of Applied Probability., 2021.

\bibitem{KnoPro2001}
A.~Knopfmacher and H.~Prodinger.
\newblock A simple card guessing game revisited.
\newblock  \textit{Electronic Journal of Combinatorics}, 8, R13:9 pages, 2001.

\bibitem{KSTY2023}
T.~Krityakierne, P.~Siriputcharoen, T.~A. Thanatipanonda, and C.~Yapolha.
\newblock \href{https://arxiv.org/abs/2209.04276}{No-feedback card guessing
  game: Moments and distributions under the optimal strategy}.
\newblock Manuscript (Arxiv), 2022.

\bibitem{NFNW-KT2022}
T.~Krityakierne and T.~A. Thanatipanonda.
\newblock \href{https://arxiv.org/abs/2205.08793}{No feedback? {N}o worries!
  {T}he art of guessing the right card}.
\newblock Manuscript (Arxiv), 2022.

\bibitem{KT2023}
T.~Krityakierne and T.~A. Thanatipanonda.
\newblock \href{https://arxiv.org/abs/2107.11142}{The card guessing game: A
  generating function approach}.
\newblock  \textit{Journal of Symbolic Computation}, 115:1--17, 2023.

\bibitem{PK2023}
M.~Kuba and A.~Panholzer.
\newblock \href{https://arxiv.org/abs/2303.04609}{On card guessing with two
  types of cards}.
\newblock Manuscript (Arxiv), 2023.

\bibitem{KuPanPro2009}
M.~Kuba, A.~Panholzer, and H.~Prodinger.
\newblock Lattice paths, sampling without replacement, and limiting
  distributions.
\newblock  \textit{Electronic Journal of Combinatorics}, 16 (1), R67:12 pages,
  2009.

\bibitem{Leva1988}
K.~Levasseur.
\newblock How to beat your kids at their own game.
\newblock  \textit{Mathematical Magazine}, 61:301--305, 1988.

\bibitem{Liu2021}
P.~Liu.
\newblock On card guessing game with one time riffle shuffle and complete
  feedback.
\newblock  \textit{Discrete Applied Mathematics}, 288:270--278, 2021.

\bibitem{OttoliniSteiner2022}
A.~Ottolini and S.~Steinerberger.
\newblock \href{https://arxiv.org/abs/2211.09094}{Guessing cards with complete
  feedback}.
\newblock Manuscript (Arxiv), 2022.

\bibitem{OT2023}
A.~Ottolini and R.~Tripathi.
\newblock \href{https://arxiv.org/abs/2303.15601}{Central limit theorem in
  complete feedback games}.
\newblock Manuscript (Arxiv), 2023.

\bibitem{Prodinger2004}
H.~Prodinger.
\newblock The kernel method: a collection of examples.
\newblock  \textit{S\'eminaire Lotharingien de Combinatoire}, 50, B50f:19
  pages, 2001.

\bibitem{Read1962}
R.~C. Read.
\newblock Card-guessing with information. a problem in probability.
\newblock  \textit{American Mathematical Monthly}, 69:506--511, 1962.

\bibitem{Zagier1990}
D.~Zagier.
\newblock How often should you beat your kids?
\newblock  \textit{Mathematical Magazine}, 63:89–92, 1990.

\end{thebibliography}

%\section*{General case: n different colors}
%\subsection*{Decomposition}
%\subsection{Hitting times}

\end{document}